\documentclass[12pt]{article}
\usepackage[utf8]{inputenc}
\usepackage{appendix}
\usepackage{xcolor}
\usepackage{graphicx}
\usepackage{pdflscape}
\usepackage{geometry}
\usepackage{verbatim}
\usepackage{float}
\newcommand{\Prob}{\mathbb{P}}
\newcommand{\Exp}{\mathbb{E}}
\newcommand{\R}{\mathbb{R}}
\newcommand{\C}{\mathbb{C}}

\newcommand{\iunit}{\mathbf{i}}

\newcommand{\x}{\mathbf{x}}
\newcommand{\ba}{\mathbf{a}}
\newcommand{\X}{\mathbf{X}}
\newcommand{\Z}{\mathbf{Z}}
\newcommand{\BSigma}{\mathbf{\Sigma}}
\newcommand{\BA}{\mathbf{A}}
\newcommand{\y}{\mathbf{y}}
\usepackage{amsthm,amsmath,amssymb,amsfonts}
\DeclareMathOperator{\Tr}{Tr}
\DeclareMathOperator{\diag}{diag}
\DeclareMathOperator{\diff}{d}

\newtheorem{definition}{Definition}[section]
\newtheorem{theorem}{Theorem}[section]
\newtheorem{corollary}{Corollary}[section]

\newtheorem{lemma}{Lemma}[section]
\theoremstyle{remark}
\newtheorem{remark}{Remark}
\title{Eigenvalue statistics of Elliptic Volatility Model with power-law tailed volatility}
\date{\today}
\author{Maltsev, Anna and Malysheva, Svetlana}
\begin{document}
\maketitle

\begin{abstract}
In this paper we study an ensemble of random matrices called Elliptic Volatility Model, which arises in finance as models of stock returns. This model consists of a product of independent matrices $X = \Sigma Z $ where $Z$ is a $T$ by $S$ matrix of i.i.d. light-tailed variables with mean 0 and variance 1 and $\Sigma$ is a diagonal matrix. In this paper, we take the randomness of $\Sigma$ to be i.i.d. heavy tailed. We obtain an explicit formula for the empirical spectral distribution of $X^*X$ in the particular case when the elements of $\Sigma$ are distributed as Student's t with parameter 3. We furthermore obtain the distribution of the largest eigenvalue in more general case, and we compare our results to financial data.

\end{abstract}

\section{Introduction}
A key problem in random matrix theory is understanding eigenvalue properties when the matrix dimensions are large. There is a large body of work on properties of the Sample Covariance Matrix Ensemble. The eigenvalue distribution has been shown to be Marchenko-Pastur in a very general case for i.i.d. variables with a variance, and then similar results were extended to matrices with various correlation structures. In this paper, we explore a random matrix ensemble originating from financial mathematics where the entries have a variance and are uncorrelated but are not independent. This dependence structure and their somewhat heavy tails result in a different eigenvalue density in the limit of large dimension. Unlike Marchenko-Pastur distribution, the eigenvalue density has a heavy tail as well. A similar eigenvalue density has been found in multiple applied fields, including in calcium imaging data in various types of tissue \cite{korovsak2019random, norris2023meaningful}, in machine learning \cite{PhysRevE.108.L022302}, and in finance \cite{plerou2002random}. The breadth of applications where such distributions are found may indicate a new universal phenomenon.

In financial mathematics, a volatility process is commonly defined as
$$
X_t=\sigma_t Z_t,
$$
where $Z_t$ are independent $\mathcal{N}(0, 1)$ random variables (further, we will call them noise), and variables $Z_t$ and $\sigma_t$ are independent. The process $\sigma_t$ is called volatility and can be modelled in multiple ways.
For example, in the celebrated Black--Scholes model, centred log-returns of the price can be modelled as
$
X_t= \sigma_0\left(B_t-B_{t-1}\right),
$
where $B_t$ is Brownian motion, the volatility equals $\sigma_0$ for any $t.$

A random variable $X$ is called heavy-tailed (or fat-tailed, Pareto-tailed, power-law tailed) with if a power law can approximate its density $p(x)$ for the large $x:$
\begin{equation}
    p(x) \underset{|x|\to \infty}{\sim} \frac{c}{|x|^{\alpha+1}}.
\end{equation}
For arbitrary $\alpha$ such tails are regularly varying, and $\alpha$ is referred to as tail exponent. A canonical example of a heavy-tailed distribution with tail exponent $\alpha$ is the Student's t distribution with $\alpha$ degrees of freedom, which we abbreviate as Student($\alpha$) throughout the manuscript.
In his work \cite{Mandelbrot1963certain}, Mandelbrot argued using the example of cotton price changes that the empirical distribution of price changes is better approximated by an  $\alpha$-stable distribution other than Normal, i.e. a distribution with tail parameter $\alpha$. The log-returns of stock prices in developed countries are believed to follow a power law with exponent $\alpha\approx 3$ \cite{gabaix_theory_2003}. This tail property is called cubic law. Figure \ref{fig:cubic_law} illustrates the cubic law for log-returns of three examples of major companies.

\begin{figure}[h]
    \centering
    \includegraphics[width=14cm]{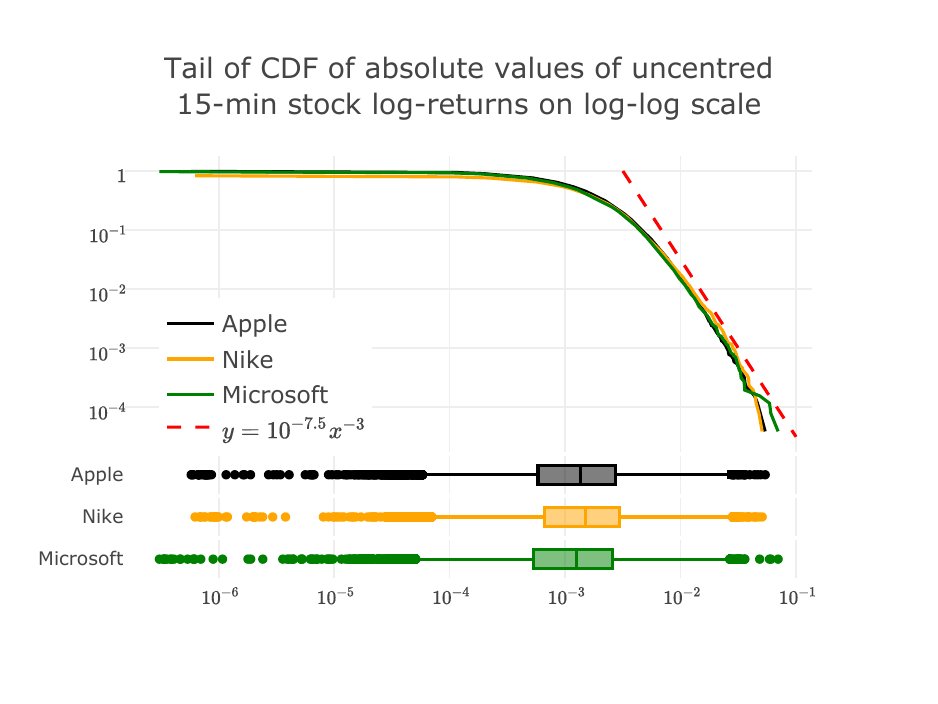}

    \caption{Illustration of the cubic law. The tail $\overline{F}(x)= 1- F(x),$ where $F(x)$ is empirical c.d.f. of returns of the chosen stock has slope $\approx 3$ when plotted on a log-log scale. The box-whiskers plot displays the distribution of logarithms of log returns. On the left, plot for three major companies. \textit{Data from polygon.io.}}
    \label{fig:cubic_law}
\end{figure}

The heaviness of the tails in stock log-returns is important for portfolio optimisation. If risky assets returns are i.i.d. and the second moment exists, investing equally into each asset, i.e. diversifying the portfolio reduces risks, and the distribution of portfolio returns can be approximated using the Central Limit Theorem (see e.g., \cite{Samuelson1967GeneralPT}). The diversification strategy may not remain optimal for the distributions with heavier tails \cite{Fama1965PortfolioAI}, \cite{Samuelson1967EfficientPS}. For example, in the case of Cauchy distributed price changes, the diversified portfolio will have a similar risk distribution as the non-diversified because the sample mean of the i.i.d. Cauchy random variables has Cauchy distribution with the same parameters. Considering even heavier tails would lead to the optimality of the non-diversification of the portfolio \cite{Ibragimov2015}.

For the sequence of random matrices $\left\{\X_T\right\}$ of the size $S\times T, $ where $\frac{S}{T}\to y>0,$ with i.i.d. entries with $0$ mean and variance $\sigma^2,$ the limiting spectral distribution of $\frac{\X_T \X_T^*}{T}$ exists, and is called Marchenko-Pastur law. This holds for heavy-tailed variables as well with tail parameter $\alpha >2$, and the limiting spectral distribution changes only for $\alpha < 2$. However, it is well-known that Marchenko-Pastur law does not approximate the spectrum of a stock returns correlation matrix, even though the returns have a tail parameter close to 3, i.e. a lot bigger than 2. This discrepancy occurs because of correlations or dependence between stocks. In \cite{Papanicolaou2016factor}, it is demonstrated that a factor model with any number of factors (a model with $k$ rank one matrices added to an i.i.d. random matrix) does not approximate stock return correlation eigenvalues either.

A model that does approximate the stock returns correlation eigenvalues well is the Student-Wishart Elliptic Volatility Matrix \cite{Biroli2007student}. In this paper we will be concerned with models that generalise this. We introduce a definition here:
\begin{definition}\label{d:elvol}
Let $T\times S$ random matrix $\X$ is an \textbf{Elliptic Volatility Matrix (EVM)} if
 \begin{equation}\label{e:X}\X = \left(\sigma_t Z_{t,s}\right)_{\substack{t\leq T\\ s\leq S}},\end{equation}
where random variables $Z_{t, s}$ are independent identically distributed random variables with a finite variance and
$\sigma_t$'s are independent of $\left(Z_{t, s}\right)_{\substack{t\leq T\\ s\leq S}}$ are random variables whose empirical cumulative distribution function converges almost surely to $F(x), $ which is the c.d.f. of  some heavy-tailed random variable  $\mathcal{\sigma}$ with tail exponent $\alpha.$ We denote $F^\prime(x) =: f(x).$ Furthermore, we define the  \textbf{Elliptic Volatility Sample Covariance Ensemble (EVSCE)} as the following random matrix ensemble
\begin{equation}\label{e:defA}
    \BA:=\frac{\X^* \X}{T},
\end{equation}
when $T, S \to +\infty$ and $\frac{T}{S} \to y,$ where $0<y<+\infty.$ Denote $\mathbf{Z} := \left(Z_{t, s}\right)_{\substack{t\leq T\\ s\leq S}}$ and $\mathbf{\Sigma} = \diag\left(\sigma_t\right)_{t\leq T}.$
\end{definition}
\noindent Notice, that in the above definition, $\X = \mathbf{\Sigma} \mathbf{Z}$ for the diagonal matrix $\mathbf{\Sigma}$ with $\mathbf{\Sigma}_{tt}=\sigma_t$ and then
\begin{equation}
    \frac{\X^* \X}{T} = \frac{\mathbf{Z}^*\mathbf{\Sigma}^2 \mathbf{Z}} {T}.
\end{equation}
\noindent The Student-Wishart is defined with $\sigma_t$'s i.i.d. Student-t distributed and $Z_{ts}$ i.i.d. Normally distributed.

The EVSCE has many limitations. It cannot fully describe the market data, as it is well-known that meaningful stock correlations, e.g. stocks in similar industries, account for some of the largest eigenvalues \cite{plerou2002random}. The discrepancy between EVSCE and market data was demonstrated definitively in \cite{Bouchard2012notelliptical} using a copula method. In \cite{Chicheportiche2013ANF}, a second volatility factor is suggested. Nevertheless, understanding the spectrum of EVSCE can be valuable as it helps elucidate mechanisms by which of the large eigenvalues of a correlation matrix can arise via dependence  and heavy tails in the distribution of the entries. Two limitations that could be relaxed in a future work are volatility clustering (“large changes tend to be followed by large changes, of either sign, and small changes tend to be followed by small changes” \cite{Mandelbrot1963certain}) and the “leverage effect” (negative past returns tend to increase future volatilities and positive past returns tend to decrease future volatilities). The ``leverage effect" could be studied via a study of dependence in the $\sigma_t$'s. Volatility clustering is already accounted for in EVSCE as the spectrum of $\frac{\X^* \X}{T}$ is preserved under the permutations of the rows of $\X$ but a reasonable model design for $\sigma_t$'s with dependence is left for future work.

The first object of study in this paper is the empirical spectral measure of an EVSCE. For a Hermitian $N \times N$ matrix $\BA$ with eigenvalues $\lambda_1, \lambda_2 \dots \lambda_N$, the probability measure $\mu_\BA$ is called its empirical spectral measure (ESM) if
\begin{equation}
    \mu_\BA := \frac{1}{N}\sum_{i=1}^{N} \delta_{\lambda_i}.
\end{equation}
The corresponding c.d.f.
\begin{equation}
    F^\BA(x): = \frac{1}{N}\#\left\{j\leq N: \lambda_j \leq x \right\}
\end{equation}
is called empirical spectral distribution (ESD) of matrix $\BA.$ Here, $\#E$ denotes the cardinality of the set $E.$
If for a given sequence of Hermitian matrices $\left\{\BA_N\right\}$
\begin{equation}
    F^{\BA_N}(x) \to F(x),
\end{equation}
for all $x\in \R$
where $F(\cdot)$ is a c.d.f. of probability measure $\mu,$ $F(\cdot)$ is called limiting spectral distribution (LSD) of this sequence, and $F^\prime(\cdot)$ is called limiting spectral density.
Note, that if limiting spectral density exists, re-normalised to probability scale histograms of matrices spectrum converge to limiting density curve when bins length are decaying to $0.$

The first result of this paper is a computation of the limiting spectral density for the EVSCE with $\sigma_t$ i.i.d. with Student(3) distribution and a general distribution $Z_{ts}$ provided it has finite moments. Our approach is via the Stieltjes transform, which for a probability measure $\mu(x)$ on the real line is defined as
\begin{equation}
    m_{\mu}(z):=\int_{-\infty}^{\infty} \frac{1}{x-z} \diff \mu(x).
\end{equation}
The statement \cite[Theorem~4.3]{BS10} provides a formula for the Stieltjes transform of the limiting density of Elliptic Volatility Model (without any requirements on volatility empirical moments convergence), which can be reduced to \eqref{eqn:sttransf3} for the case of the Student(3) volatility. It is known, that for the measure $\mu$ on the real line with density function $\rho(x)$ for $x\in \R$
\begin{equation}\label{e:invtrans}
    \rho(x)= \frac{1}{\pi} \lim_{\varepsilon\to 0^{+}} \Im m_\mu(x+\iunit\varepsilon).
\end{equation}
Using \eqref{eqn:sttransf3}  we carefully follow the construction of the solution of a quartic polynomial to find the solution with imaginary part in $\C^{+}$ and to show that it is unique. Our construction furthermore allows us to obtain the exact expression for the limit its imaginary part when approaching the real line. This new approach to solving a self-consistent equation for a Stieltjes transform directly, using a carefully constructed solution whose imaginary part is then easy to understand, could be useful in finding explicit limiting densities in other random matrix ensembles.

Our second object of study in this paper is the largest eigenvalue in an EVSCE.
In the simplest case, when $\mathbf{\Sigma} = \mathbf{1}$, the position of the largest eigenvalue depends on the existence of the fourth moment of the distribution of $Z_{ts}$. When the fourth moment is finite, the largest eigenvalue will converge to the spectral edge \cite{bai_no_1998}. When the fourth moment does not exist, the largest eigenvalue will diverge \cite{bai_note_1988}. Multiple studies regarding the $k$th largest eigenvalue were conducted when entries of $\Z$ regularly vary with exponent $0<\alpha<4.$
Soshnikov \cite{Soshnikov2005PoissonSF} gives the weak limit in the case of the Cauchy distribution of the matrix entries. Biroli et al. \cite{Biroli2007TopEV}, using physical methods, extend this result to regularly varying tails with index $0 <\alpha < 4$. Auffinger-Ben Arous-Peche\cite{Auffinger2009PoissonConvergence} prove rigorously that a point process based on the largest eigenvalues of $\X \X^*$ converges to a Poisson point process depending on the distribution of the entries.
Their results were extended by allowing dependence in the entries of $\X,$ for example, in \cite{Davis2011LimitTF} $X_{i t}=\sum_j c_j Z_{i, t-j}$, with regularly varying noise $\left(Z_{i t}\right)$ with tail index $\alpha \in(0,4)$. The obtained point process would then depend on $\sum c_j^2$ as well. Other examples of estimating the largest eigenvalue after implementing a dependence structure in heavy-tailed sample covariance matrices include \cite{Davis2016ExtremeVA}, \cite{Heiny2016EigenvaluesAE}, \cite{Basrak2019ExtremeES}, \cite{Heiny2019TheEO}, \cite{Teimouri2020AsymptoticBO}. Additionally, in \cite{Janssen2016TheEO}, under the condition that the dimension $S$ is fixed and $T\to \infty,$ two cases with non-linear dependence are discussed: the tail of volatility dominates the tail of noise, and the tail of noise dominates the tail of volatility.

Understanding largest eigenvalues is important for applications in data science and machine learning, particularly as they pertain to Principal Component Analysis (PCA). In PCA, principal components are eigenvectors that correspond to the largest eigenvalues of a sample covariance matrix, and the largest eigenvalue shows the amount of variability of the data set that the principal component captures. Due to the BBP transition \cite{Baik2005phase}, one can often deduce meaningful information about the data from the largest eigenvalue and corresponding eigenvector. However, heavier-tails in the data, for example a diverging 4-th moment of the matrix entries, can lead to anomalously large eigenvalues with no information content (see example in \cite{Biroli2007TopEV}). In multiple examples of heavy-tailed time series with or without dependence, the largest eigenvalues are
essentially determined by the extreme order statistics from an array of i.i.d. random
variables. The extreme eigenvalues' asymptotic behaviour follows from the classical extreme value theory. Thus understanding behaviours of largest eigenvalues in a heavy-tailed setting can improve our understanding of conditions for effective use of PCA.

We will prove that under appropriate scaling the largest eigenvalue in EVSCE is approximated by the square of the maximal value of the $\sigma_t$'s and we find the scaling constant. Many of the papers mentioned above follow a common methodological blueprint. The first step is to show that the matrix $\X \X^{*}$ is well approximated by its diagonal. The second step is to derive the extremes of the diagonal of $\X \X^{*}.$ The largest eigenvalue is usually close to the maximum of some identically distributed random variables. This way, the largest eigenvalue of $\X\X^*$ is similar to the largest entry of $\X\X^*,$ usually found on the diagonal, and it is also similar to the square of the largest entry of $\X.$ In this paper we will adopt a similar methodological approach. Furthermore, we will estimate the error, then use simulation data to illustrate our theoretical results. We will also compare the largest eigenvalue of the data matrix with the largest eigenvalue of simulations.

Lastly, we perform explicit data analytics to illustrate our results via simulations and to compare them to real-world financial returns data.
Suppose that $S_t(\operatorname{Open})$ and $S_t(\operatorname{Close})$ denote the open and close prices of the stock on the $t$-th time interval. We are interested in log-return of the price on time interval $t,$ defined as
\begin{equation}
    X_t:= \log \frac{S_t(\operatorname{Close})}{S_{t}(\operatorname{Open})}.
    \label{eqn:log-return}
\end{equation}
We directly study the distribution of returns at a given time $t$, compute its standard deviation as an estimate of $\sigma_t$. Then we observe that the tail parameter of the $\sigma_t$ is approximately 3. We also note that the plots of spillovers for the EVSCE look like those from data. We further compare a simulation of the EVSCE to the data and to our analytic results, both for the maximum eigenvalue and for the limiting density.

This paper is organised as follows. In Section \ref{sec:density}, we obtain the explicit expression for the limiting spectrum of EVSCE in the case of  Student(3)-distributed volatility. In Section \ref{sec:maximaleig}, we study the maximal eigenvalue of EVSCE when the volatility's tail exponent is $0<\alpha<4.$ In Section \ref{sec:data}, we apply compare the EVSCE and our analytic results to a data matrix obtained from 15-min S\&P stock prices log-returns.

\section{Spectral properties of Elliptic Volatility matrix}
\label{sec:density}
The Stieltjes transform of limiting spectral distribution of matrix $\BA$ can be obtained, using the following simplification of \cite[Theorem~4.3]{BS10}.
\begin{theorem}
\label{thm:bsstieltjes}
Suppose that the entries of $\mathbf{Y}(n\times p)$ are complex random variables that are independent for each $n$ and identically distributed for all $n$ and satisfy $\mathrm{E}\left(\left|Y_{11}-\mathrm{E}\left(Y_{11}\right)\right|^2\right)=1$. Also, assume that $\mathbf{T}=\operatorname{diag}\left(\tau_1, \ldots, \tau_p\right)$, $\tau_i$ is real, and the empirical distribution function of $\left\{\tau_1, \ldots, \tau_p\right\}$ converges almost surely to a probability distribution function $H$ as $n \rightarrow \infty$. Set $\mathbf{B}:=\frac{1}{n} \mathbf{Y} \mathbf{T} \mathbf{Y}^*.$  Assume also that $\mathbf{Y}$ and  $\mathbf{T}$ are independent. When $p=p(n)$ with $p / n \rightarrow y>0$ as $n \rightarrow \infty$, then, almost surely, $F^{\mathbf{B}_n}$, the ESD of the eigenvalues of $\mathbf{B}_n$, converges vaguely, as $n \rightarrow \infty$, to a (nonrandom) d.f. $F$, where for any $z \in \mathbb{C}^{+} \equiv\{z \in \mathbb{C}: \Im z>0\}$, its Stieltjes transform $s=s(z)$ is the unique solution in $\mathbb{C}^{+}$ to the equation
$$
s=\frac{1}{y \int \frac{\tau d H(\tau)}{1+\tau s}-z}.
$$
\end{theorem}

\begin{remark}

\begin{enumerate}
\item There is no requirement on the moment convergence of the empirical spectral distribution of $\mathbf{T}, $ thus $H$ can have any regularly varying tail.

\item While in Lemma \ref{l:Stieltjes} and Theorem \ref{t:density} we introduce an assumption of independence on $\sigma_t$'s we only use it for the application of Theorem \ref{thm:bsstieltjes}, which does not require independence. Thus this condition could potentially be relaxed for sequences of $\sigma_t$ such that the empirical distribution function of $\left\{\tau_1, \ldots, \tau_p\right\}$ converges almost surely to a Student(3).
\end{enumerate}
\end{remark}

The Stieltjes transform for the EVSCE model was obtained in \cite{Biroli2007student} in an integral form for a general Student's t distribution. Here we obtain an explicit expression of the Stieltjes transform in the particular case of the Student(3). The result follows directly from the Theorem given above.
\begin{lemma}\label{l:Stieltjes}
For $\mathbf{X}$ as in Definition \ref{d:elvol} with $\sigma_t$ distributed as independent Student(3) for all $t$, the Stieltjes transform of the limiting spectral distribution is given by \begin{equation}\label{e:Stieltjes}
\frac{1}{s(z)}+z=\frac{1}{\left(1+\sqrt{\frac{s(z)}{y}}\right)^2}.
\end{equation}

\end{lemma}

\begin{proof}
Matching the notation in Theorem \ref{thm:bsstieltjes}, we set $\mathbf{Y}:= \Z^*,$ $\mathbf{T}:=\BSigma^2,$ $n:=S,$ and $p:=T$ then
the theorem gives the Stieltjes transform of the matrix
$$
\frac{\X^* \X }{S} = \frac{T}{S}\BA = \frac{\Z^* \BSigma^2 \Z}{S},
$$
and in this case $y: = \lim_{T\to \infty} \frac{T}{S}.$

Let $s_0(z)$ be the limiting Stieltjes transform of $\frac{\mathbf{X}^* \mathbf{X}}{S},$ and $s(z)$ the limiting Stieltjes transform of $\frac{\mathbf{X}^* \mathbf{X}}{T}.$ Then
\begin{equation*}
    y s_0(yz) = s(z).
\end{equation*}
By Theorem \ref{thm:bsstieltjes}
\begin{equation*}
    s_0(z) = \frac{1}{y\int\frac{\tau\diff H(\tau)}{1+\tau s_0(z)}-z}.
\end{equation*}
Therefore,
\begin{equation}
    s(z) = y s_0(yz) = \frac{y}{y\int\frac{\tau\diff H(\tau)}{1+\tau s_0(yz)}-yz} = \frac{1}{\int\frac{\tau\diff H(\tau)}{1+\tau s_0(yz)}-z} = \frac{1}{\int\frac{\tau\diff H(\tau)}{1+\frac{\tau}{y}s(z)}-z}.
    \label{eqn:sttransfours}
\end{equation}

We will rewrite equation \eqref{eqn:sttransfours} for the case when the volatility has re-normalised Student($\nu$) with $\nu>2$ degrees of freedom.

The probability density function of standard Student($\nu$)
\begin{equation}
g_\nu(t)=\frac{\Gamma\left(\frac{\nu+1}{2}\right)}{\sqrt{\nu \pi} \Gamma\left(\frac{\nu}{2}\right)}\left(1+\frac{t^2}{\nu}\right)^{-(\nu+1) / 2}.
\end{equation}
It has $0$ mean and variance $\frac{\nu}{\nu-2}.$ The density of re-normalised Student($\nu$) (standard Student($\nu$) divided by $\sqrt{\frac{\nu}{\nu-2}}$) is

\begin{equation}
    f_{\nu}(t):= \sqrt{\frac{\nu}{\nu-2}}g_\nu\left(\sqrt{\frac{\nu}{\nu-2}}t\right)= \frac{\Gamma\left(\frac{\nu+1}{2}\right)}{\sqrt{(\nu-2) \pi} \Gamma\left(\frac{\nu}{2}\right)}\left(1+\frac{t^2}{\nu-2}\right)^{-(\nu+1) / 2}.
\end{equation}

The diagonal elements of $\mathbf{\Sigma}^2$ are distributed as the squared re-normalised Student's t distributed random variable, therefore the empirical distribution of diagonal elements of $\mathbf{\Sigma}^2$ has limiting density $h_\nu(\tau),$ that we will find below. Let $F_\nu(\cdot)$ be the c.d.f. of re-normalised Student's t distribution, and $H_\nu(\cdot)$ be the c.d.f. of the diagonal elements of $\mathbf{\Sigma}^2.$ For $\tau>0$ holds
$$
H_\nu(\tau) = F_\nu\left(\sqrt{\tau}\right) - F_\nu\left(-\sqrt{\tau}\right).
$$
Thus,
\begin{multline*}
h_\nu(\tau) = \frac{1}{2\sqrt{\tau}} \left(f_\nu\left(\sqrt{\tau}\right) +f_{\nu}\left(-\sqrt{\tau}\right)\right)\\ =\frac{\Gamma\left(\frac{\nu+1}{2}\right)}{\sqrt{(\nu-2) \pi} \Gamma\left(\frac{\nu}{2}\right)}\left(1+\frac{\tau}{\nu-2}\right)^{-(\nu+1) / 2} \times \frac{1}{\sqrt{\tau}},
\end{multline*}
for $\tau>0.$ Particularly, for $\nu=3$ we can compute
\begin{equation}
    h_3(\tau) = \frac{2}{\pi}(1+\tau)^{-2} \times \frac{1}{\sqrt{\tau}}.
\end{equation}
Equation \eqref{eqn:sttransfours} yields
\begin{equation}
  \frac{1}{s(z)}+z=\int_0^{+\infty} \frac{\tau h_3(\tau)\diff \tau}{1+\tau\frac{s(z)}{y}}
  =
  \frac{2}{\pi}\int_0^{+\infty}\frac{\sqrt{\tau} \, }{\left(1+\tau\right)^2\left(1+\tau\frac{s(z)}{y}\right)} \diff \tau
  =
   \frac{1}{\left(1+\sqrt{\frac{s(z)}{y}}\right)^2}
  \label{eqn:sttransf3}
\end{equation}
where the principal branch cut of the square root is taken.
\end{proof}

While the tail asymptotic of the Stieltjes transform is given in equation (11) of \cite{Biroli2007student}, in the following corollary we offer a simple proof in the case of Student(3) for volatilities: \begin{corollary}\label{c:tail3}
Let $\rho(x)$ be the limiting density of eigenvalues in the EVSCE with i.i.d. Student(3)-distributed $\sigma_t$'s. Then the tail asymptotic is given by
\begin{equation}
\lim_{x \rightarrow \infty} \frac{\rho(x)}{x^{2.5}} = \frac{2}{\sqrt{y}\pi}
\end{equation}
\end{corollary}
\begin{proof}
First we observe that since the branch cut of the square root is principal and thus has a positive real part,
\begin{equation}
\left|\frac{1}{1+\sqrt{\frac{s(z)}{y}}}\right| \leq 1.
\end{equation}
Thus for large $x$, equation \eqref{e:Stieltjes} implies that
\begin{equation}
\frac{1}{s(x + i0^+)} = -x + o(x)
\end{equation}
which yields that $\Re s(x + i0^+) = -1/x + o(1/x)$ as well as that $|s(x + i0^+)| = \frac{1}{x} + o(1/x)$, which furthermore implies that $\Im s(x + i0^+) = o(1/x)$. This implies that $\arg(\sqrt{s(x + i0^+)})$ is near $\pi/2$. Now from equation \eqref{e:Stieltjes} we see that
\begin{equation}
\frac{-\Im s(x+i0^+)}{|s(x+i0^+)|^2} = \Im \frac{1}{1+ \frac{s}{y} + 2\sqrt{\frac{s(x+i0^+)}{y}}}= -2 \Im \sqrt{\frac{s(x+i0^+)}{y}}+ o\left(\frac{1}{x}\right)= \frac{2}{\sqrt{yx}} + o\left(\frac{1}{\sqrt{x}}\right)
\end{equation}
yielding that
$
\Im s(x+i0^+) = \frac{2}{x^{2.5}\sqrt y} + o(1/x^{2.5})
$
and via equation \eqref{e:invtrans} we obtain the corollary.
\end{proof}

\subsection{Derivation of the limiting density when $\nu=3.$}
Here we offer a derivation of the limiting density for EVSCE with Student(3) volatilities.
\begin{theorem}\label{t:density}
Let \begin{multline}\label{e:q}
q := y^6 (x-1)^6+6 y^5 (x-1)^3 \left(x^2+4 x+1\right)
\\
+3 y^4 \left(5 x^4+16 x^3+30 x^2+16 x+5\right)+3 y^2 \left(5 x^2+2 x+5\right)
\\+4 y^3 \left(6 x^{3/2} \sqrt{3 y^3 (x-1)^3+9 y^2 \left(x^2+7 x+1\right)
+9 y (x-1)+3}+
5 x^3+12 x^2-12 x-5\right)
\\
+6 y (x-1)+1,
\end{multline} and let $w_*$ be given by
\begin{multline}\label{e:w*}
12 x^2w_*:= \\
-y^2 \left(x^2+10 x+1\right)-
2 \sqrt[3]{q} +2 y (x-1)+1
\\
-\frac{2 \left(y^4 (x-1)^4+4 y^3 \left(x^3+3 x^2-3 x-1\right)+6 y^2 (x+1)^2+4 y (x-1)+1\right)}{\sqrt[3]{q}}.
\end{multline}
Furthermore let
\begin{equation}\begin{split}\label{e:ABC}
&A:=-\frac{y^2 \left(x^2+10 x+1\right)+2 y (x-1)+1}{2 x^2} \\
&B:=-\frac{4 y^3 (1 + x)}{x^2}\\
&C:=\frac{y^4 (x+1)^2 \left(x^2-14 x+1\right)+4 y^3 (x-1) (x+1)^2+6 y^2 (x+1)^2+4 y (x-1)+1}{16 x^4}.
\end{split}\end{equation}
and
let $R^{\pm} \in \mathbb{R}$ be given by
\begin{equation}\begin{split}\label{e:Rpm}
R^+ &:= 2 w_* - A\\
R^- &:= - 2w_* - A.
\end{split}\end{equation}
For $\mathbf{X}$ as in Definition \ref{d:elvol} with $\sigma_t$ distributed as independent Student(3) for all $t$, the limiting density of eigenvalues for $x>0$ of $\mathbf{A}$ as in \eqref{e:defA} is given by
\begin{equation}\label{e:rho}
\rho(x) = \frac{1}{2\pi}\sqrt{-R^- - \frac{2B}{\sqrt{R^+}}}.
\end{equation}
\end{theorem}

\begin{proof}[Proof of Theorem \ref{t:density}]
By equation \eqref{eqn:sttransf3}, the limiting density $\rho(x) = \lim_{\eta \downarrow 0} \Im s_*$ where $s_*$ has positive imaginary part and is the solution of the equation derived above in \eqref{e:Stieltjes}.
To find the solution $s_*$ we rewrite the equation as follows:
\begin{equation}
\sqrt{\frac{4 s }{y}} = \frac{s}{s z+1}-\left(\frac{s}{y}+1\right)
\end{equation}
Now we square both sides and multiply through by the denominator to obtain a quartic polynomial
\begin{equation}\label{e:quartic}
Q(s):=\frac{4 s (s z+1)^2}{y} - \left(s-\left(\frac{s}{y}+1\right)(s z+1)\right)^2 = 0.
\end{equation}
When we do this, we will introduce spurious solutions. We will first demonstrate that these spurious solutions are real for all values of $z>0$ and $y > 1$.

The spurious solutions will satisfy the following equation:
\begin{equation}
\sqrt{\frac{4 s }{y}} = -\left(\frac{s}{s z+1}-\left(\frac{s}{y}+1\right)\right)
\end{equation}
equivalent to
\begin{equation}\label{s:spurious}
\sqrt{\frac{4 s }{y}} -\frac{s}{y}= -\frac{s}{s z+1} +1.
\end{equation}
We notice that the RHS is a parabola in
$\sqrt{s/y}$ with zeros at 0 and $\frac{2}{\sqrt y}$ and maximum at 1. The left hand size is 1 at 1 and is strictly decreasing to
$1 - \frac{1}{z}$ as  $s\rightarrow \infty$. Thus there are two real solution to equation \eqref{s:spurious}
in the interval $(0, \frac{2}{\sqrt y})$ for any $z>0$.

The quartic equation was first solved by Cardano and Ferrari in 1540. Here we follow a more modern construction of the solution to a quartic polynomial using a resolvent cubic equation, see e.g. Theorem 4 in \cite{chavez2022complete}. Throughout this proof we use Mathematica to assist with labour-intensive computations, and our Mathematica notebook is attached to this manuscript. We know from algebra that a quartic polynomial has exactly two complex solutions if and only if its discriminant is negative. As we have shown that for $z>0$, $Q$ has real solutions, we deduce that when the discriminant is positive, Q has 4 real solutions and thus no solution with positive imaginary part. To find the spectral edge, it suffices to find $z>0$ where the discriminant is negative. Taking the discriminant of $Q$ we obtain
\begin{multline}
\text{Disc}(Q) = \\
-\frac{256}{y^6} \left(y^3 z^6-3 y^3 z^5+3 y^3 z^4-y^3 z^3+3 y^2 z^5+21 y^2 z^4+3 y^2 z^3+3 y z^4-3 y z^3+z^3\right)
\end{multline}
yielding the following equation, after division by common factors,
\begin{equation}\label{e:disceq}
1 - 3 y + 3 y^2 - y^3 + 3 y z + 21 y^2 z + 3 y^3 z + 3 y^2 z^2 -
 3 y^3 z^2 + y^3 z^3=0.
\end{equation}
For $y>0$, this equation has the following solutions
\begin{align}
z &= \frac{\left(\sqrt[3]{y}-1\right)^3}{y} \label{e:spEdge}
\\
  \text{
 or  } z &= \frac{3 y^{2/3}-3 \sqrt[3]{y}+2 y-2}{2 y} \pm i\frac{3 \sqrt{3} \left(\sqrt[3]{y}+1\right)}{2 y^{2/3}}  \label{e:othersol}
\end{align}
Noting that \eqref{e:othersol} has a non-zero imaginary part for all $y >0$ we deduce that \eqref{e:spEdge} yields the spectral edge.

Now we proceed to construct the solution of the quartic with positive imaginary part. First we transform the quartic $Q$ into a monic depressed quartic $\tilde Q$ via
\begin{multline} \label{e:depressed}
\tilde Q(s) = -\frac{y^2}{z^2} Q\left(s - \frac{1}{4} \left(-\frac{2 y}{z}-2 y+\frac{2}{z}\right)\right)
= s^4 + As^2 + Bs + C \\
\end{multline}
where we have set $A, B, C$ as in \eqref{e:ABC}.
We now construct and solve the resolvent cubic equation
\begin{equation}
P(w):= (2 w - A)(w^2 - C) - B^2/4=0.
\end{equation}
Recalling the equations \eqref{e:q} and \eqref{e:w*} for $q$ and $w_*$, we note that $w_*$ is a real solution of the above equation.

We notice that $w_*$ is indeed real whenever $q$ is real and $q$ is real whenever
\begin{equation}
3 y^3 (z-1)^3+9 y^2 \left(z^2+7 z+1\right)
+9 y (z-1)+3 >0
\end{equation}
We notice that this inequality is identical to \eqref{e:disceq} and is satisfied whenever $z$ is above the spectral edge. Thus $q$ and $w_*$ are real whenever $z$ is above the spectral edge.

Let $R^{\pm}$ be as in \eqref{e:Rpm}. Then for $ j, k \in \{0, 1\}$ the four solutions of the depressed quartic equation $\tilde Q$ are given by
\begin{equation}
2s = (- 1)^j \sqrt{ R^+} + (-1)^k \sqrt{R^- + (-1)^{j+1} \frac{2B}{\sqrt{R^+}}}
\end{equation}
We will prove that $R^+ >0$ for $z$ above the spectral edge using the established fact that exactly two of the solutions are real.

Suppose for contradiction that $R^+ < 0$. We take the standard branch cut of the square root along the negative $x$-axis, with $\sqrt{-1} = i$, making $\sqrt {R^+}$ purely imaginary with positive imaginary part. Recall also that $B < 0$ for $z, y > 0$, making $\Im \left(\frac{2B}{\sqrt{R^+}}\right) > 0$ . Then if $(- 1)^j \sqrt{ R^+} + (-1)^k \sqrt{R^- + (-1)^{j+1} \frac{2B}{\sqrt{R^+}}} \in \mathbb{R}$, we must have $j = k$, which would yield two real solutions. This would imply that the two solutions with $j \neq k$ must be complex conjugates. However taking the complex conjugate of the solution with $j=0 = k-1$ we check that its conjugate does not equal the solution with $j = 1 = k+1$:
\begin{equation}
\overline{\sqrt{ R^+} -  \sqrt{R^- - \frac{2B}{\sqrt{R^+}}}} = - \sqrt{ R^+} -  \sqrt{R^- +\frac{2B}{\sqrt{R^+}}}
\neq -\sqrt{ R^+} +  \sqrt{R^- + \frac{2B}{\sqrt{R^+}}}
\end{equation}
where for the last statement we recall that $R^- \in \mathbb{R}$ while $\frac{2B}{\sqrt{R^+}}$ would be purely imaginary. Thus by contradiction we have established that $R^+ >0$ and thus $\sqrt{R^+} >0$.

Thus $2\Im s = \Im \left( \sqrt{R^- \pm \frac{2B}{\sqrt{R^+}}} \right)$. We recall again that the solutions form exactly one conjugate pair, implying that one of $R^- \pm \frac{2B}{\sqrt{R^+}}$ is positive and the other is negative. As $B < 0$, we deduce that $R^- + \frac{2B}{\sqrt{R^+}} < 0,$
yielding that
\begin{equation}
\Im s = \frac{1}{2}\sqrt{-R^- - \frac{2B}{\sqrt{R^+}}}.
\end{equation}
We notice that $-R^- - \frac{2B}{\sqrt{R^+}}$ is continuous in $z$ as a complex variable for $z$ strictly above the spectral edge, thus the identity $\rho(x) = \frac{1}{\pi}\lim_{\eta \downarrow 0}\Im s(x+i\eta)$ yields the desired result.
\end{proof}

\section{Statistics of the maximal eigenvalue}
\label{sec:maximaleig}
In this part of the work we investigate the distribution of the rescaled maximal eigenvalue in the Elliptic Volatility Model. We notice that here we need a stronger assumption on the moments of the matrix entries $Z_{ts}$ from equation \eqref{e:X}. Suppose that $N\times N$ Hermitian matrix $\mathbf{H}$ has spectrum $\lambda_1, \dots \lambda_N.$ We denote
\begin{equation}
    \lambda_{\max}\left[\mathbf{H}\right]:= \max_{1\leq i \leq N} \left|\lambda_i\right|.
\end{equation}
Note, that $\lambda_{\max}\left[\cdot\right]$ is a norm on a linear space of $N\times N$ Hermitian matrices, i.e. it is positive for non-zero matrices and $\lambda_{\max}\left[\mathbf{H}_1+\mathbf{H}_2\right]\leq \lambda_{\max}\left[\mathbf{H}_1\right]+\lambda_{\max}\left[\mathbf{H}_2\right].$
We prove the following theorem
\begin{theorem}
Suppose $0<\alpha<4, $  $\alpha \neq 2.$ Let $\mathbf{\Sigma}$ be a $T\times T$ diagonal matrix i.i.d. power-law tailed diagonal entries $\sigma_i$ with tail exponent $\alpha.$ Suppose that $\mathbf{X}$ is a $T\times S$ matrix, where $\frac{T}{S}\to y>0,$ independent of $\mathbf{\Sigma}, $ whose entries are i.i.d., have $0$ mean and variance $1,$ and have all moments.
Denote as $H(\cdot)$ a c.d.f. of random variable $\sigma_i^2.$
Let $a_T$  be a solution of $1-H(a_T)=\frac{1}{T}.$
Then
\begin{equation}\frac{\lambda_{\max}\left[\mathbf{\Sigma} \mathbf{Z} \mathbf{Z}^* \mathbf{\Sigma}\right]}{Sa_T} \underset{d}{\to}{\xi},
\end{equation}
where $\Prob(\xi\leq x) = \exp\left(-x^{-\frac{\alpha}{2}}\right).$
\label{thm:max_eigenvalue}
\end{theorem}
Note, that since the matrix $\mathbf{\Sigma} \mathbf{Z} \mathbf{Z}^* \mathbf{\Sigma}$ is positive semi-definite, all its eigenvalues are non-negative and therefore, $\lambda_{\max}\left[\mathbf{\Sigma} \mathbf{Z} \mathbf{Z}^* \mathbf{\Sigma}\right]$ is the maximal eigenvalue of this matrix.
\begin{remark}
\label{rem:simplewayoforder}
The spectrum of the matrices $\mathbf{\Sigma} \mathbf{Z} \mathbf{Z}^* \mathbf{\Sigma}$ and $\mathbf{\Sigma}^2 \mathbf{Z} \mathbf{Z}^*$ is the same.  Therefore, Marčenko-Pastur bounds (\cite{yin_limit_1988}) yield that the order of the maximal eigenvalue does not exceed $O(\max_{1\leq i\leq T}\sigma_i^2).$ Nevertheless, it does not yield that the limiting distribution of rescaled maximal eigenvalue of $\mathbf{\Sigma} \mathbf{Z} \mathbf{Z}^* \mathbf{\Sigma}$ is Fréchet on that particular scale.
\end{remark}
\begin{remark}
\label{rem:aTorder}
The distribution of $\sigma_i^2$ is regularly varying with exponent $\frac{\alpha}{2},$ therefore $a_T = O\left(T^{2/\alpha}\right).$
\end{remark}
\begin{remark}
\label{rem:aTstudent}
If $\sigma_i$ are distributed as re-normalized Student(3), $a_T$ can be found directly as the solution of the equation
\begin{equation}
1-H_3(a_T) = \frac{
2\left(-\frac{\sqrt{a_T}}{1+a_T}+\operatorname{ArcTan}\left[\frac{1}{\sqrt{a_T}}\right]\right)}{\pi}=\frac{1}{T}.
\end{equation}
\end{remark}

We consider 2 cases: $0<\alpha<2$ and $2<\alpha<4.$
In the first case, we estimate the maximal eigenvalue with maximum of diagonal elements from the bottom and with $\lVert \mathbf{\Sigma} \mathbf{Z} \mathbf{Z}^* \mathbf{\Sigma}\rVert_{\infty} $ from the top. In the second case, we estimate the norm of the matrix formed by non-diagonal elements first. For $1\leq i \leq T$ we will denote as $\mathbf{z}_i$ the vectors, made by the rows of matrix $\mathbf{Z}$ and we will denote the scalar product of the vectors $\x$ and $\y$ as $\left<\x, \y\right>.$ The length of the vector $\x$ we will denote as $\lVert \x \rVert_2.$ For any matrix $\mathbf{M}$ we will denote the matrix, formed by its diagonal elements as $\diag\left[\mathbf{M}\right].$

The outline of the proof is the following. First of all, we prove that
\begin{equation}
    \frac{\max_{i\leq T}\left(\left(\mathbf{\Sigma} \mathbf{Z} \mathbf{Z}^* \mathbf{\Sigma}\right)_{ii}\right)}{Sa_T} \underset{d}{\to}{\xi}.
    \label{eqn:max_to_frechet}
\end{equation}
where $\xi$ is the Fréchet distribution. Afterwards, we prove that
\begin{equation}
    \frac{\left|\lambda_{\max} \left[\mathbf{\Sigma} \mathbf{Z} \mathbf{Z}^* \mathbf{\Sigma} - \diag \left(\mathbf{\Sigma} \mathbf{Z} \mathbf{Z}^* \mathbf{\Sigma}\right) \right]\right| }{S T^{2/\alpha}}\underset{d}{\to} 0.
    \label{eqn:error}
\end{equation}
Then we obtain that

\begin{equation}
\lambda_{\max}\left[\mathbf{\Sigma} \mathbf{Z} \mathbf{Z}^{*} \mathbf{\Sigma}\right] \sim \max_{i}\left(\mathbf{\Sigma} \mathbf{Z} \mathbf{Z}^* \mathbf{\Sigma}\right)_{ii},
\end{equation}
and, subsequently, the Theorem \ref{thm:max_eigenvalue} will follow. The proof of convergence \eqref{eqn:max_to_frechet} is similar for both cases $0<\alpha<2$ and $2<\alpha<4.$ In the proof of convergence \eqref{eqn:error} each case is considered separately.
\begin{lemma} Fix $x>0.$ For matrices $\mathbf{\Sigma}$ and $\mathbf{Z}$ as above
    \begin{equation}
        \lim _{T \rightarrow \infty} \mathbb{P}\left(\frac{\max_{i\leq T}\left(\left(\mathbf{\Sigma} \mathbf{Z} \mathbf{Z}^* \mathbf{\Sigma}\right)_{ii}\right)}{Sa_T} \leq x\right)=\lim _{T \rightarrow \infty} \mathbb{P}\left(\frac{\max_{i\leq T}\sigma_i^2 }{a_T} \leq x\right)
        \end{equation}
\label{lemma:diag_frechet}
\end{lemma}

\begin{proof}
Rewrite
\begin{equation}
    \frac{\left(\mathbf{\Sigma} \mathbf{Z} \mathbf{Z}^* \mathbf{\Sigma}\right)_{ii}}{S a_T} = \frac{\sigma_i^2}{a_T}\times \frac{\lVert \mathbf{z}_i\rVert_2^2}{S}.
\end{equation}
Fix $\varepsilon >0 .$ By Corollary \ref{col:a_diag}
    \begin{equation}       \Prob\left(\left|\frac{\max\lVert \mathbf{z}_i\rVert^2}{S}-1\right|>\varepsilon\right)\underset{T\to +\infty}{\to} 0.
    \end{equation}
Therefore,
    \begin{equation}
    \Prob\left(\left|\frac{\max_{i}\left[d_ i^2\lVert\mathbf{z}_i\rVert_2^2\right]}
        {\max_i \sigma_i^2 S}-1\right|>\epsilon\right) \to 0.
    \end{equation}
We conclude, that for all $x>0, \varepsilon>0$
\begin{multline}
\lim _{T \rightarrow \infty} \mathbb{P}\left(\frac{\max_{i\leq T}\left(\left(\mathbf{\Sigma} \mathbf{Z} \mathbf{Z}^* \mathbf{\Sigma}\right)_{ii}\right)}{Sa_T}\leq x(1+\epsilon)\right) \\ \leq\lim _{T \rightarrow \infty} \mathbb{P}\left(\frac{\max_{i\leq T}\sigma_i^2 }{a_T} \leq x\right) \leq \lim _{T \rightarrow \infty} \mathbb{P}\left(\frac{\max_{i\leq T}\left(\left(\mathbf{\Sigma} \mathbf{Z} \mathbf{Z}^* \mathbf{\Sigma}\right)_{ii}\right)}{Sa_T} \leq x(1-\epsilon)\right),
\end{multline}
which leads to the statement of the Lemma.
\end{proof}

Now, combining Lemma \ref{lemma:diag_frechet} and Lemma \ref{lemma:frechet} it is enough to prove, the following
\begin{corollary}

When $\sigma_i$ are regularly varying with exponent $0<\alpha<4$
\begin{equation}
     \frac{\lambda_{\max} \left[\mathbf{\Sigma} \mathbf{Z} \mathbf{Z}^* \mathbf{\Sigma} - \diag \left(\mathbf{\Sigma} \mathbf{Z} \mathbf{Z}^* \mathbf{\Sigma}\right) \right]}{Sa_T} \underset{d}{\to} 0.
\end{equation}
\label{lemma:nondiad_eig_to0}
\end{corollary}

We prove by cases $0<\alpha<2$ and $2<\alpha<4.$

\subsection{Proof of Lemma \ref{lemma:nondiad_eig_to0} for $ 0<\alpha<2$}
By Lemma \ref{lemma:inftynorm},
\begin{multline}
\lambda_{\max} \left[\mathbf{\Sigma} \mathbf{Z} \mathbf{Z}^* \mathbf{\Sigma} - \diag \left(\mathbf{\Sigma} \mathbf{Z} \mathbf{Z}^* \mathbf{\Sigma}\right) \right] \leq
\left\lVert \mathbf{\Sigma} \mathbf{Z} \mathbf{Z}^* \mathbf{\Sigma} - \diag \left(\mathbf{\Sigma} \mathbf{Z} \mathbf{Z}^* \mathbf{\Sigma}\right)\right\rVert_\infty \\ \leq \max_{1\leq i \leq T} \left[ |\sigma_i| \sum_{j \neq i} |\sigma_j| |\left< \mathbf{z}_i, \mathbf{z}_j\right>|\right] \leq \max_{1\leq i \leq T} |\sigma_i| \sum_{j=1}^{T}|\sigma_j| \max_{i\neq j} |\left< \mathbf{z}_i, \mathbf{z}_j\right>|.
\end{multline}
Therefore, we need to prove, that
\begin{equation}
    \frac{\max_{1\leq i \leq T} |\sigma_i| \sum_{j=1}^{T}|\sigma_j| \times \max_{i\neq j} |\left< \mathbf{z}_i, \mathbf{z}_j\right>|}{ST^{2/\alpha}} \underset{T\to +\infty}{\to}0.
\end{equation}
Note, that similarly to Lemma \ref{lemma:frechet},
\begin{equation}
    \frac{\max_{1\leq i \leq T} |\sigma_i|}{T^{1/\alpha}} = O(1).
\end{equation}
By Corollary \ref{col:a_nondiag} for all $\epsilon>0$
\begin{equation}
    \frac{\max_{i\neq j} |\left< \mathbf{z}_i, \mathbf{z}_j\right>|}{S^{1/2+\varepsilon}} \to 0.
\end{equation}
It is enough to show, that for some $\epsilon>0.$
\begin{equation}
    \frac{\sum_{1\leq i \leq T} |\sigma_i|}{T^{1/\alpha+1/2 -\varepsilon}} = O(1).
\end{equation}
If $1<\alpha<2$ by the Law of Large Numbers
\begin{equation}  \frac{\sum_{1\leq i \leq T} |\sigma_i|}{T} \underset{d}{\to} \Exp|\sigma_i|,
\end{equation}
therefore for $\epsilon<\frac{1}{\alpha}-\frac{1}{2}$
\begin{equation}
\frac{\sum_{1\leq i \leq T} |\sigma_i|}{T^{1/\alpha+1/2 -\varepsilon}} \underset{d}{\to} 0.
\end{equation}

By \cite{bingham_regular_1989}[Corollary 8.1.7]  for $0<\alpha<1$ and $\lambda>0$
\begin{equation}
    g_{|\sigma_i|}(\lambda) :=\Exp[e^{-\lambda |\sigma_i|}]= 1 - C \lambda^{\alpha} + o\left(\lambda^{\alpha}\right),
\end{equation}
therefore
\begin{equation}  \Exp\left[\exp\left(-\lambda \frac{\sum_{1\leq i \leq T} |\sigma_i|}{T^{1/\alpha+1/2 -\varepsilon}}\right)\right]= \left(1 - C \frac{\lambda^{\alpha}}{T^{1+\alpha\left(1/2 -\varepsilon\right)}} + o\left(\frac{\lambda^{\alpha}}{T^{1+\alpha\left(1/2 -\varepsilon\right)}}\right)\right)^T \underset{T\to +\infty}{\to} 1,
\end{equation}
which yields
\begin{equation}
    \frac{\sum_{1\leq i \leq T} |\sigma_i|}{T^{1/\alpha+1/2 -\varepsilon}} \underset{d}{\to} 0.
\end{equation}

\subsection{Proof of Lemma \ref{lemma:nondiad_eig_to0} for $2<\alpha<4$}

\begin{proof}
Note, that it is enough to prove that
\begin{equation}
    \frac{\Tr\left[\mathbf{\Sigma} \mathbf{Z} \mathbf{Z}^* \mathbf{\Sigma} - \diag \left(\mathbf{\Sigma} \mathbf{Z} \mathbf{Z}^* \mathbf{\Sigma}\right) \right]^2}{S^2 T^{4/\alpha}} \underset{d}{\to } 0.
\end{equation}
We can expand
\begin{equation}
    \left[\mathbf{\Sigma} \mathbf{Z} \mathbf{Z}^* \mathbf{\Sigma} - \diag \left(\mathbf{\Sigma} \mathbf{Z} \mathbf{Z}^* \mathbf{\Sigma}\right) \right]_{ij} =
\begin{cases}
    \sigma_i \sigma_j \left< \mathbf{z}_i, \mathbf{z}_j\right> & \text{if } i\neq j \\
    0 &\text{if } i=j.
\end{cases}
\end{equation}
Therefore,
\begin{equation}
    \left[\mathbf{\Sigma} \mathbf{Z} \mathbf{Z}^* \mathbf{\Sigma} - \diag \left(\mathbf{\Sigma} \mathbf{Z} \mathbf{Z}^* \mathbf{\Sigma}\right) \right]^2 =: \mathbf{T},
\end{equation}
    where $\mathbf{T}$ is such that
    \begin{equation}
        \mathbf{T}_{ij}:=
        \begin{cases}
            \sigma_i^2 \sum_{k:k\neq i} d^2_k \left< \mathbf{z}_i, \mathbf{z}_k\right>^2 &\textit{if }i=j\\
            \sigma_i \sigma_j \sum_{\substack{k\neq i\\ k\neq j} } \sigma_k^2 \left< \mathbf{z}_i, \mathbf{z}_k\right>\left< \mathbf{z}_j, \mathbf{z}_k\right>  &\textit{if } i\neq j.
        \end{cases}
    \end{equation}
    Notice, that
    \begin{equation}
        \Tr \mathbf{T} \leq \left(\sum_{i=1}^T \sigma_k^2\right)^2 \max_{i\neq j}\left<\mathbf{z}_i, \mathbf{z}_j\right>^2,
    \end{equation}
    Therefore, combination of Law of Large Numbers for $\sum_{k=1}^T\sigma_k^2$ and Corollary \ref{col:a_nondiag} yields that for all $\varepsilon>0$ $$\frac{\lambda_{\max}\left[\mathbf{T}\right]}{S^{1+\varepsilon}T^2} \underset{d}{\to } 0. $$
    The convergence above for $\varepsilon<\frac{4}{\alpha}-1$ yields the statement of the Lemma.
    \end{proof}
    \begin{remark}
    \label{rem:errorbottom}
    One can see that
    \begin{equation}
        \mathbf{T}= \sum_{j} \sigma_k^2 \mathbf{R}^{(k)},
    \end{equation}
    where
    \begin{equation}
        \mathbf{R}^{(k)}_{ij}:=
        \begin{cases}
        d_{i}^2 \left< \mathbf{z}_k, \mathbf{z}_i\right>^2 &\textit{if } i= j \neq k, \\
        d_{i} d_{j} \left< \mathbf{z}_k, \mathbf{z}_i\right>\left< \mathbf{z}_k, \mathbf{z}_j\right>
        &\textit{if }i\neq j\neq k\\
        0 &\textit{if }i=k \textit{ or }j=k
        \end{cases}
    \end{equation}
    Notice, that
    \begin{equation}
        \mathbf{R}^{(k)} = \mathbf{v}_k \mathbf{v}_k^T,
    \end{equation}
    where $\mathbf{v}_k$ is a $T$-dimensional column vector with
    \begin{equation}
        \left(\mathbf{v}_k\right)_i :=
        \begin{cases}
            \sigma_i \left< \mathbf{z}_k, \mathbf{z}_i\right> &\textit{if }i\neq k \\
            0  &\textit{if }i=k.
        \end{cases}
    \end{equation}
    Independently of $\mathbf{\Sigma}, $ we can choose $\frac{T}{2}$ random vectors $\mathbf{z}_i $ that are lying in the same $T$-dimensional semi-space and thus their pairwise scalar products are positive, and denote their set as of indices as $I_0.$  Denote the set of indices $i$ such that $\sigma_i>0$ as $I_1,$ and take $I =I_0\cap I_1.$ Suppose for the vector $\mathbf{w}, $ for $i\in I$ $w_i = 1$ and $w_i = 0$ otherwise.
    \begin{equation}
        \mathbf{w}^T \mathbf{T} \mathbf{w} \geq \sum_{i\in I} \sigma_i^2 \left(\sum_{k\in I, k\neq i} \sigma_k \left<\mathbf{z}_k, \mathbf{z}_i\right>\right)^2.
    \end{equation}
        Since $\left<\mathbf{z}_i, \mathbf{z}_k\right> \sim \sqrt{S}, $
    \begin{equation}
         \mathbf{w}^T \mathbf{T} \mathbf{w} \sim \frac{1}{4^3} T^3 S,
    \end{equation}
        thus, it is the exact order of the error can be estimated as
    $$
    \lambda_{max}[\mathbf{T}] \gtrsim \frac{1}{16} T^2 S
    $$
    and
    $$\lambda_{\max}\left[\mathbf{\Sigma} \mathbf{Z} \mathbf{Z}^* \mathbf{\Sigma} - \diag \left(\mathbf{\Sigma} \mathbf{Z} \mathbf{Z}^* \mathbf{\Sigma}\right)\right]\gtrsim \frac{1}{4} T S^{1/2}.
    $$
    \end{remark}

    Furthermore, we show the significance of the error and provide the numerical simulations comparing the largest eigenvalues of the submatrices of the cleared data and the largest eigenvalues in the EVSCE. Figure \ref{fig:modelvsfrechet} shows that on the scale of dimentionality of our data, the error is still significant. When $T= \frac{T_{data}}{50},$ the Remark \ref{rem:aTstudent} yields, that $a_T\approx 36.281.$  This way, $S=485$ and $T = 512.$
    We can see, that by Remark \ref{rem:errorbottom}
    \begin{equation}
        \frac{|\lambda_{\max} \left[\mathbf{\Sigma} \mathbf{Z} \mathbf{Z}^* \mathbf{\Sigma} - \diag \left(\mathbf{\Sigma} \mathbf{Z} \mathbf{Z}^* \mathbf{\Sigma}\right) \right]|}{S a_T} \gtrsim \frac{\frac{1}{4} T S^{1/2}}{Sa_T}= \frac{\frac{1}{4} T }{\sqrt{S} a_T}\approx 0.16
    \end{equation}
    This error plays important role when $\max_{i\leq T}\left(\mathbf{\Sigma} \mathbf{Z} \mathbf{Z}^* \mathbf{\Sigma}\right)_{ii}$ takes values that are close to 0.
    \begin{figure}[H]
    \centering
    \includegraphics[width=0.6\textwidth]{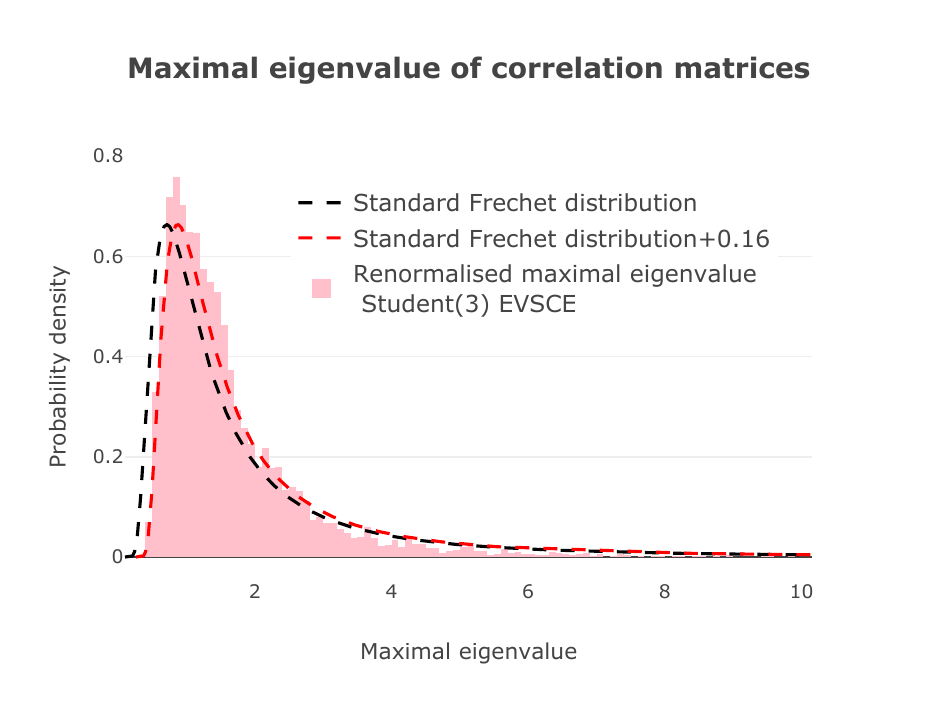}
        \caption{Histogram of the maximal eigenvalue in a simulated EVSCE when $\sigma_i$ has Student(3) distribution. The dimensions of the matrix simulated matrix are $S=485, T=512.$ Number of simulations is $N=5000.$}
        \label{fig:modelvsfrechet}
    \end{figure}
    It explains why the histogram of numerical simulations does not completely match the limiting distribution. Figure \ref{fig:modelvsfrechet} shows that for the matrices modelled with these dimensions the distribution $F_{\lambda_{\max}}$ of the renormalised largest eigenvalue is
    \begin{equation}
        F_{\xi+0.64}(x) < F_{\lambda_{\max}} (x) < F_{\xi+0.16}(x),
    \end{equation}
     where $\xi$ has Fréchet distribution.


\section{Comparison of EVSCE and historical data}
\label{sec:data}
We conduct research on S\&P500  15-minute intervals of stock-returns from January 2020 to October 2022.  Data were obtained from \textit{polygon.io}.

\subsection{Data preparation: removing the “market mode” and re-normalisation}
The return of the stock $s$ over the time interval $t$ is calculated in the following way:

\begin{equation}
    \mathbf{x}_{ts}^\prime := \log \left(\frac{\textit{Close price}_s(t)}{\textit{Open price}_s(t)}\right),
\end{equation}
where $\textit{Close price}_s(t)$ and ${\textit{Open price}_s(t)}$ denote Close and Open price of the stock $s$ on the time interval $t$ respectively. If there were no sales of the stock $s$ on the time interval $t$ and, consequently, Open and Close prices can not be determined, we assume the value of $\mathbf{x}_{ts}^\prime$ to be equal to $0$. Denote the matrix $\mathbf{X}^{\prime}:= \left(\mathbf{x}_{ts}^\prime\right)_{\substack{t\leq T\\ s\leq S}}.$ Below we describe the procedure of re-normalisation and “market mode” removal.

The first step is to obtain the matrix $\mathbf{X}_{data},$ is from the matrix $\mathbf{X}^{\prime}$ with the procedure of re-normalization described below.

For the $T \times S$ matrix $\mathbf{Y}$ the procedure of re-normalization conducted the following way:
\begin{itemize}
    \item For each entry of the matrix $\mathbf{Y}$ subtract the empirical mean of the entries of its column.
    \item Divide each entry of the matrix you got in the previous step by the empirical standard deviation of the entries of its column.
    \end{itemize}
This way, the re-normalized matrix will have on the intersection of the row $t$ and column $s$ the number
\begin{equation}
    y_{ts}^\prime := \frac{y_{ts} - \overline{y}_s}{\sigma_s},
\end{equation}
where
\begin{equation}
    \overline{y}_s:= \frac{1}{T} \sum_{t=1}^{T} y_{ts}
\end{equation}
 and
\begin{equation}
    \sigma_s := \sqrt{\frac{1}{T-1}\sum_{t=1}^{T}\left(y_{ts} -  \overline{y}_s\right)^2}.
\end{equation}

The “market mode” causes the overwhelming majority of entries of the matrix $\frac{\mathbf{X}_{data}^* \mathbf{X}_{data}}{T}$ to be positive and drives its maximal eigenvalue. It also causes the maximal eigenvalue of $\frac{\mathbf{X}_{data}^* \mathbf{X}_{data}}{T}$ to be significantly larger then the typical maximal eigenvalue EVSCE with Student(3)-distributed $\sigma_t$'s.

For the second step, we apply standard PCA to separate the “market mode” of $\mathbf{X}_{data}$ and re-normalize the result.  To “clear” $T\times S$ matrix $\mathbf{Y}$ using the $S$-component vector $\mathbf{y}_0,$ such that $\lVert \mathbf{y}_0\rVert=1,$ we replace each row $\mathbf{y}_t$ of the matrix $\mathbf{Y}$ with
\begin{equation}
    \mathbf{y}_t^\prime : = \mathbf{y}_t - \left< \mathbf{y}_t,  \mathbf{y}_0\right>  \mathbf{y}_0.
\end{equation}
To separate the “market mode” we apply the procedure of “clearing” to the matrix $\mathbf{X}_{data}$ using the vector $\mathbf{x}_{\max},$ where $\mathbf{x}_{\max},$ is the eigenvector of $\frac{\mathbf{X}_{data}^* \mathbf{X}_{data}}{T}$ corresponding the maximal eigenvalue $\lambda_{\max}$. We obtain the matrix $\mathbf{X}_{cl}^\prime$, and after applying re-normalization procedure to $\mathbf{X}_{cl}^\prime$ we obtain the matrix $\mathbf{X}_{cl}.$

Note, that eigenvalues of $\frac{\mathbf{X}_{data}^* \mathbf{X}_{data}}{T}$ and $\frac{\left(\mathbf{X}^\prime_{cl}\right)^* \mathbf{X}_{cl}^\prime}{T},$ apart from $\lambda_{\max},$ are matching, and $\frac{\left(\mathbf{X}^\prime_{cl}\right)^* \mathbf{X}_{cl}^\prime}{T},$ has $0$ instead of $\lambda_{\max}.$ Nevertheless, after re-normalization, eigenvalues can shift depending on the sample variances of columns of the matrix $\mathbf{X}_{cl}^\prime.$ Further, we compare the spectrum of the matrix $\frac{\X_{cl}^*\X_{cl}}{T}$ with the spectrum of the matrix $\frac{\X^* \X}{T},$ where the matrix $\X$ is obtained from Elliptic Volatility Model and has the same size as $\X_{cl}.$

\subsection{Data analytics: comparing EVSCE and market data}
First we observe that data suggests that the distribution of empirical standard deviations of the rows of “cleared” returns is heavy-tailed (see Figure \ref{fig:tail_of_std}) and the tail parameter is approximately 3.
\begin{figure}[H]
    \centering
    \includegraphics[width=0.7\textwidth]{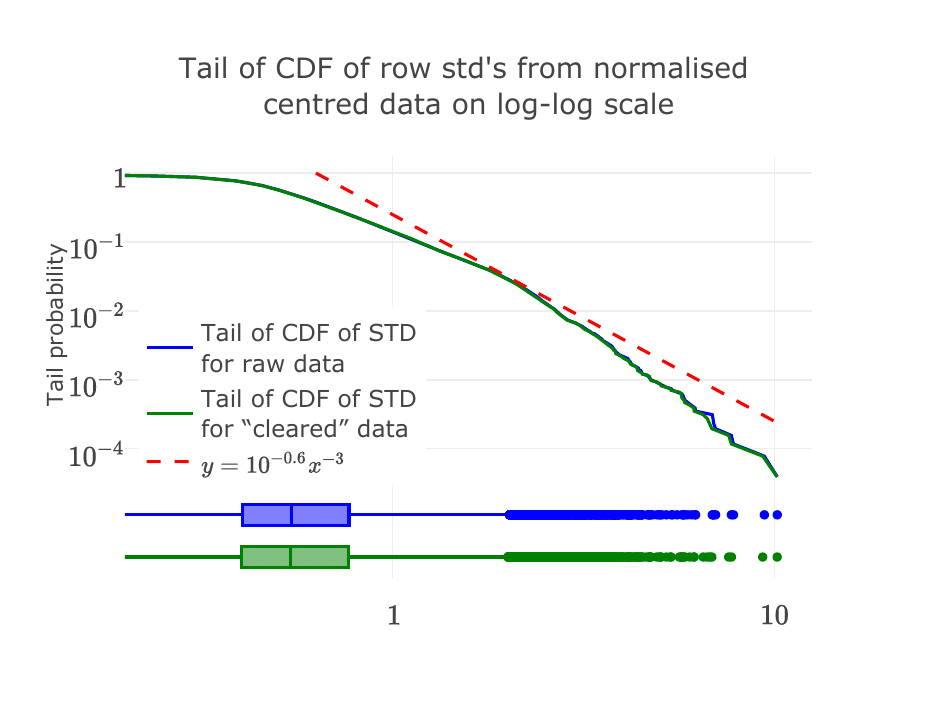}
    \caption{Cubic law for standard deviations of return vectors at fixed time. Data taken from polygon.io}
    \label{fig:tail_of_std}
\end{figure}

The Figure \ref{fig:cl_data_our} (top left) shows that the histogram of spectrum of the matrix $\frac{\X_{cl}^* \X_{cl}}{T}$ is well approximated by the limiting spectrum of EVSCE with $\sigma_t$'s i.i.d. as Student(3). Figure \ref{fig:cl_data_our} (top right) shows that it is not well approximated by the spectrum of EVSCE with Normally distributed volatility. In EVSCE the heaviness of the tail of the limiting spectrum depends on the heaviness of the distribution of the volatility as shown in equation (11) of \cite{Biroli2007student} and Corollary \ref{c:tail3}. Figure \ref{fig:cl_data_our} (bottom) shows the histograms of eigenvalues for the data and the simulation EVSCE with $\sigma_t$'s i.i.d. as Student(3) where the entries of $\mathbf{X}$ are randomly shuffled. The Marchenko-Pastur law is plotted as well, and we see that the shuffled data approximates the Marchenko-Pastur law well. This is a control to verify that the dependence and correlation structures in the two data sets cause the heavy tails in the corresponding spectral measures.

\begin{figure}[H]
    \centering
    \includegraphics[width=0.45\textwidth]{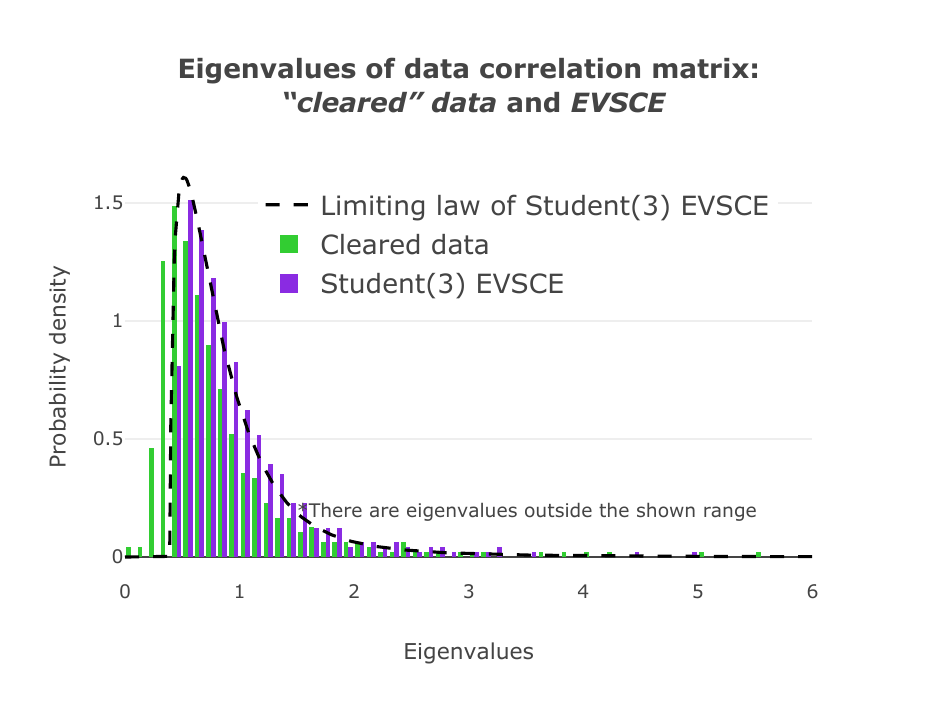}
    \includegraphics[width=0.45\textwidth]{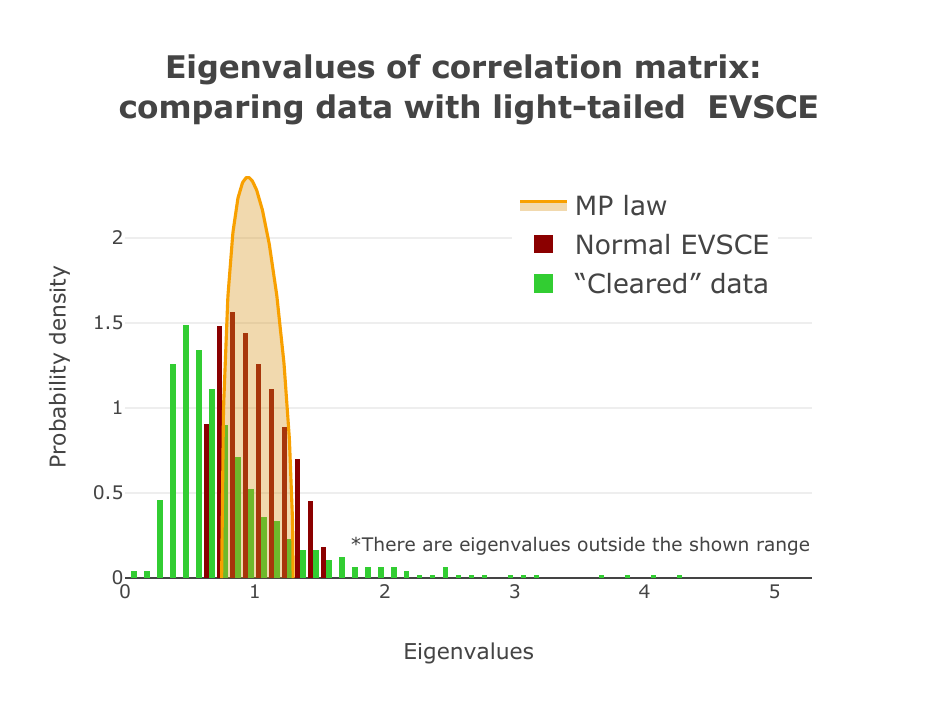}
    \includegraphics[width=0.6\textwidth]{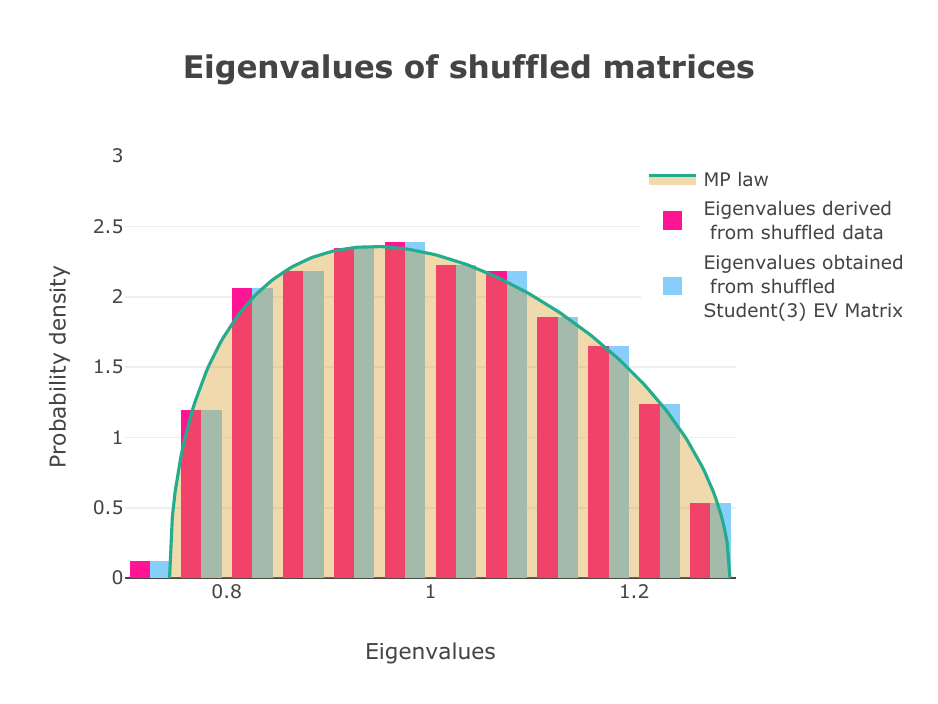}
    \caption{(Top left) Histograms of spectrum of simulated Student(3) EVSCE and matrix $\frac{\X_{cl}^*\X_{cl}}{T},$ and the limit obtained in Theorem \ref{t:density}. (Top right) Comparison of the Spectrum of $\frac{\X_{cl}^*\X_{cl}}{T},$ randomly generated EVSCE with Normally distributed $\sigma_t$'s and Marcenko-Pastur law. (Bottom) Comparison of spectrum of covariance matrix of shuffled data and similarly shuffled EVM to Marchenko-Pastur law. \textit{Data from polygon.io.}}
    \label{fig:cl_data_our}
\end{figure}

To study the distribution of the maximum eigenvalue of the data, we divide the original data into 50 equal parts, and apply the same normalisation and clearing procedure to each part as we did to the whole data set in Section \ref{sec:data}. Figure \ref{fig:m_eig_data} shows that maximal eigenvalues of the EVSCE model have the same order of magnitude as the maximal eigenvalues of sample covariance matrices obtained from re-normalised and cleared parts of $\X_{cl}^\prime.$ Nevertheless, their distribution is not a complete match.
\begin{figure}[H]
    \centering
    \includegraphics[width=0.6\textwidth]{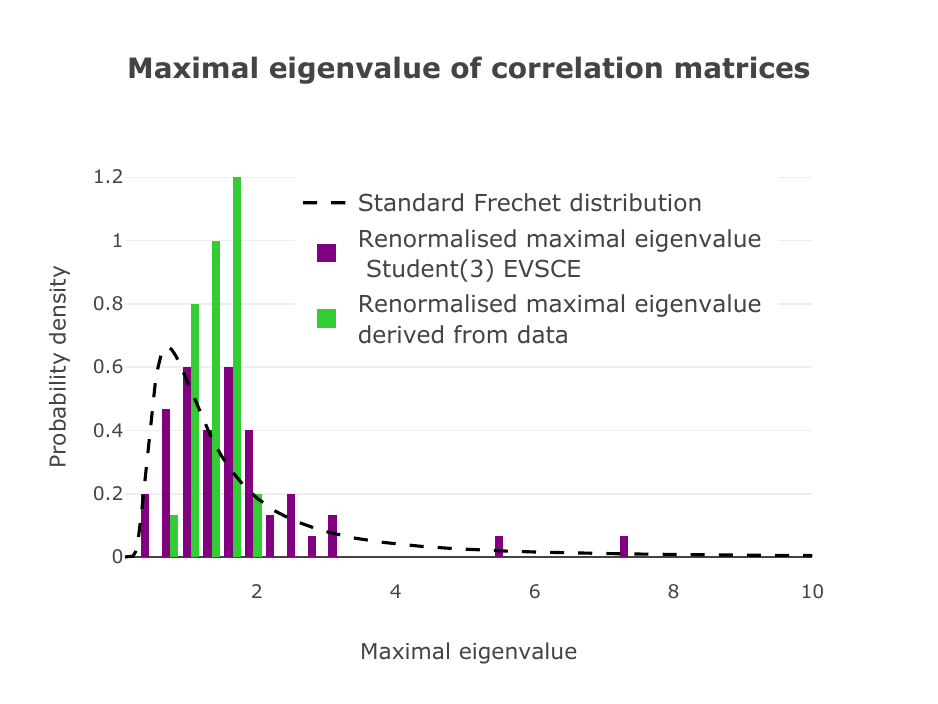}
    \caption{Histogram of re-normalised maximal eigenvalues: matrix $\X_{cl}^\prime$ were separated into 50 equal parts each part was ``cleared" (in green)(data from polygon.io). Histogram of 50 re-normalised maximal eigenvalues of Student(3) EVSCE of the same size (in purple).}
    \label{fig:m_eig_data}
\end{figure}

While some spectral properties of the stock returns covariance matrix may be due to correlations of stocks, e.g. from companies in the same economic sector, their dependence structure is also important. The EVSCE has uncorrelated but dependent random variables. Scatter plots where the returns of one stock are plotted against the returns of another can be used to show that extreme returns tend to happen simultaneously (see, e.g., \cite{aussenegg2008simple} ). Heavy tailed Elliptic Volatility random variables, that is random variables of the form of columns of the EVM, can account for such “spillovers,” that is log-returns of 2 different stocks can be “jointly heavy” (compare Figure \ref{fig:tail_dep} top left derived from data to top right which shows two EV random variables). Data suggests that the values of “spillover” pairs of returns may not be easily explained solely through the combination of “heavy-tailness” of each return and the correlation coefficient (compare Figure \ref{fig:tail_dep} top left derived from data to bottom left and right). The scatterplots from data and the EV random variables appear to have a convex 2D shape (top plots), while the ones derived from independent or correlated variables appear to have a concave shape.

\begin{figure}
    \centering
    \includegraphics[width=0.8\textwidth]{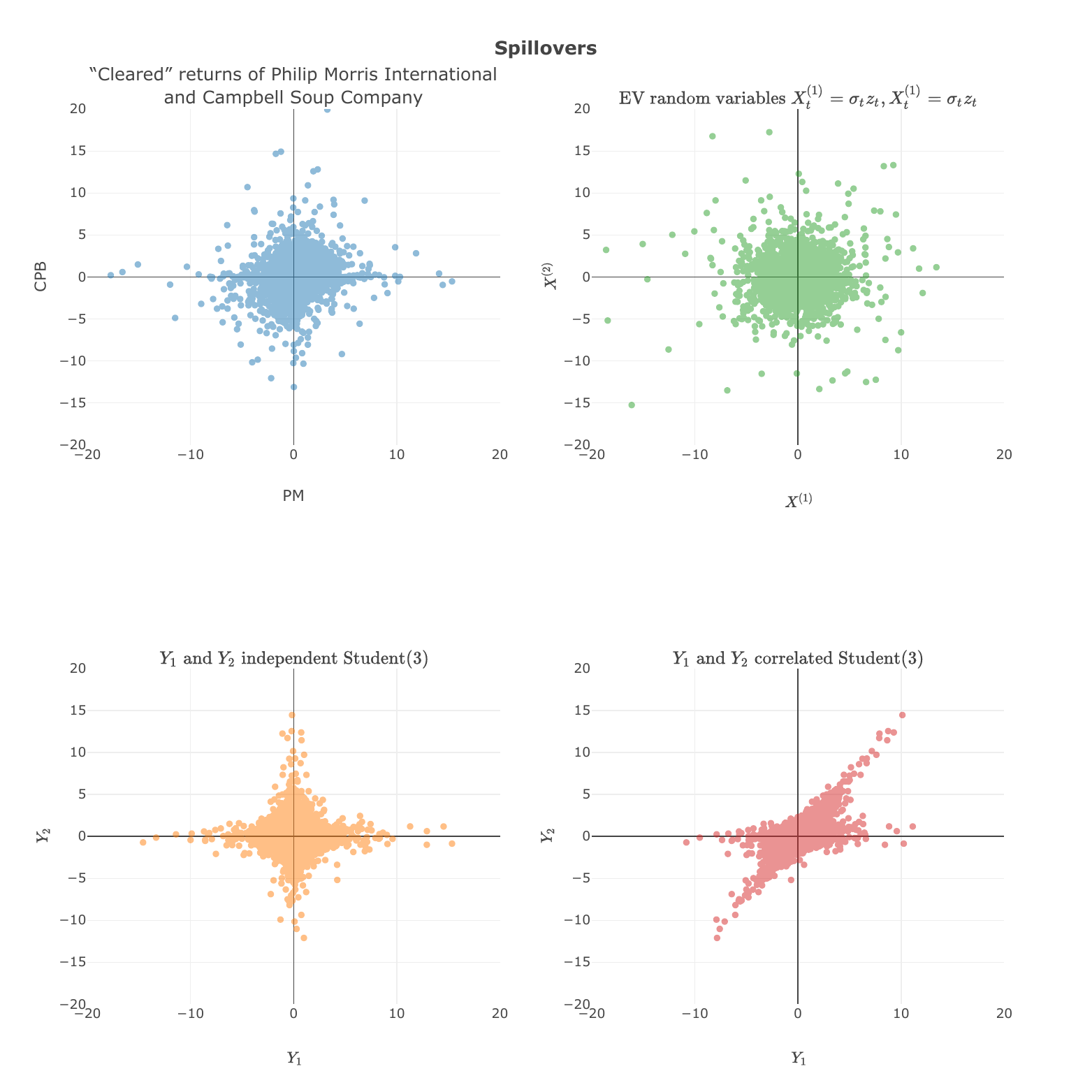}
    \caption{Joint heaviness of log-returns of two stocks may not follow solely from the heaviness of log-returns of each stock and their correlation coefficient. Top left demonstrates the scatter-plot of “cleared” log-returns of two different stocks (data from polygon.io). Top right demonstrates scatter plot of $X^{(1)}=\sigma \xi_1, X^{(2)}=\sigma \xi_2,$ where $\sigma$ is re-normalised Student(3) r.v., $\xi_1$ and $\xi_2$ are independent standard normal r.v.'s independent of $\sigma.$ Bottom left shows the scatter-plot of population derived from 2 independent re-normalised Student(3) random variables. Bottom right demonstrates the scatter-plot of $Y_1 = a_1, Y_2 = \frac{a_1+a_2}{\sqrt{2}}$, where $a_1$ and $a_2$ are independent re-normalised Student(3).}
    \label{fig:tail_dep}
\end{figure}

\begin{subappendices}
\section{Appendix}
\begin{lemma} For $N\times N$ Hermitian matrix $\mathbf{H}$ denote $$\left\lVert \mathbf{H}\right\rVert _{\infty}:= \max_{1\leq i\leq N} \sum_{j=1}^N \left|\mathbf{H}_{i, j}\right|.$$ Then
\begin{equation*}
       \lambda_{\max}\left[\mathbf{H}\right] \leq \left\lVert \mathbf{H}\right\rVert _{\infty}.
\end{equation*}
\label{lemma:inftynorm}
\begin{proof}
    For $N$-dimensional  column vector $\x$ we notice, that
\begin{multline*}
    \x^* \mathbf{H}\x = \sum_{i=1}^N \sum_{j=1}^N  x_i \mathbf{H}_{i, j} x_j \leq \sum_{i=1}^N \sum_{j=1}^N \mathbf|{H}_{i, j}| \frac{|x_j|^2+|x_i|^2}{2} \\ =\frac{ \sum_{1\leq i \leq N} |x_i|^2 \sum_{j=1}^N \mathbf|{H}_{i, j}| + \sum_{1\leq j \leq N} |x_j|^2 \sum_{i=1}^N \mathbf|{H}_{i, j}|}{2} \leq  \left\lVert \mathbf{H}\right\rVert _{\infty} \lVert \x\rVert _2^2
\end{multline*}

Therefore, for all $N$-dimensional  column vectors $\x$
\begin{equation*}
    \frac{\x^* \mathbf{H}\x}{\lVert \x\rVert_2^2} \leq \lVert \mathbf{H}\rVert _{\infty},
\end{equation*}
which yields the statement of the Lemma.
\end{proof}
\end{lemma}
\begin{theorem}
    Suppose that $\xi_1, \xi_2, \dots$ are i.i.d. random variables with 0 mean and at least $2k$ finite moments. Then for all $\varepsilon>0$
    \begin{equation}
        \Prob\left(\left|\frac{\sum_{i=1}^N \xi_i}{N}\right|>\varepsilon \right) =O(N^{-k}).
    \end{equation}
\end{theorem}
\begin{proof}
By Markov inequality
\begin{multline}
    \Prob\left(\left|\frac{\sum_{i=1}^N \xi_i}{N}\right|>\varepsilon \right) \leq\varepsilon^{-2k} \frac{\Exp\left(\sum_{i=1}^N \xi_i\right)^{2k}}{N^{2k}} \\
    = \varepsilon^{-2k} \left(\begin{array}{l} n \\ k \end{array}\right)N^{-2k} + O\left(N^{-k-1}\right) = O\left(N^{-k}\right)
\end{multline}
\end{proof}
\begin{lemma}
    Suppose that the entries of $S$-dimensional random vector $\x$  are i.i.d. random variables with first 4 moments independent of $\x$ and unit variance. Then
    \begin{equation}
        S\Prob\left(\left\lVert\frac{\lVert \x\rVert^2}{S}-1\right\rVert >\varepsilon\right) \underset{S \to +\infty}{\to} 0
    \end{equation}
\end{lemma}
\begin{corollary}
For matrix $\mathbf{Z}$ as in the statement of the Theorem \ref{thm:max_eigenvalue}
\begin{equation}
    \max_i\left|\lVert\mathbf{z}_i\rVert^2 -1\right| \underset{S\to +\infty}{\to} 0
\end{equation}
\label{col:a_diag}
\end{corollary}
\begin{lemma}
    Suppose that the entries of $S$-dimensional vectors $\x$ and $\y$ are i.i.d. random variables with the first 6 moments independent of $S.$ Then
    \begin{equation}
        S^2 \Prob \left(\frac{\left<\x, \y\right>2}{S^2}>\epsilon \right) \to 0.
    \end{equation}
\end{lemma}
\begin{corollary}
For matrix $\mathbf{Z}$ as in the statement of the Theorem \ref{thm:max_eigenvalue}
\begin{equation}
    \frac{\max_{i\neq j}|\left<\mathbf{z}_i, \mathbf{z}_j\right>|}{S^{1/2+\epsilon}} \underset{\Prob}{\to} 0.
\end{equation}
\label{col:a_nondiag}
\end{corollary}

\begin{lemma}
For matrix $\mathbf{\Sigma}$ and $a_T$ as in the statement of the Theorem \ref{thm:max_eigenvalue}
\begin{equation}
    \frac{\max_{i\leq T}\sigma_i^2}{a_T}\underset{d}{\to} \xi
\end{equation}
where $\xi$ is a random variable distributed as $\text{Fréchet}(\frac{\alpha}{2}, 1, 0)$
\label{lemma:frechet}
\end{lemma}
\begin{proof}
\begin{multline}
    \Prob\left(\frac{\max_{i\leq T}\sigma_i^2}{a_T} > x \right) = \Prob\left(\frac{\max_{i\leq T}\sigma_i^2}{a_T} > x a_T \right) \\ = 1- \Prob\left(\frac{\max_{i\leq T}\sigma_i^2}{a_T} \leq x a_T \right) = 1 -\left( 1 - H(xa_T)\right)^T
\end{multline}
   As $H(a_T) = \frac{1}{T} $ and the distribution is regularly varying, $ H(xa_T) \sim \frac{x^{-\alpha/2}}{T}.$
Therefore,
\begin{equation}
   1 -\left( 1 - H(xa_T)\right)^T \sim 1- \left(1 -\frac{x^{-\alpha/2}}{T}\right)^T \underset{T \to +\infty }{\to }1 - e^{-x^{-\alpha/2}},
\end{equation}
which is the tail of $\text{Fréchet}(\frac{\alpha}{2}, 1, 0)$ distribution.
\end{proof}

\end{subappendices}

\bibliographystyle{alpha}
\bibliography{bibliography}

\end{document}